\documentclass[12pt]{article}

\usepackage{amssymb,amsmath}
\usepackage{graphicx}
\usepackage{mathdots}
\usepackage{amsbsy}
\usepackage{amsthm}
\usepackage{mathrsfs}
\usepackage{verbatim}
\usepackage[colorlinks]{hyperref}
\usepackage{authblk}
\usepackage{fullpage}
\usepackage[english]{babel}
\usepackage[T1]{fontenc}
\usepackage[utf8]{inputenc}
\usepackage{tikz}
\usepackage{multirow}
\usetikzlibrary{calc}

\theoremstyle{definition}
\newtheorem{Theorem}{Theorem}[section]

\newtheorem{Lemma}{Lemma}[section]
\newtheorem{Proposition}{Proposition}[section]
\newtheorem{Example}{Example}[section]
\newtheorem{Remark}{Remark}[section]
\newtheorem{Definition}{Definition}[section]

\newcommand{\C}{\mathbb{C}}

\newcommand{\R}{\mathbb{R}}

\newcommand{\Z}{\mathbb{Z}}

\newcommand{\mc}[1]{\mathcal{#1}} % short for mathcal
\newcommand{\ms}[1]{\mathscr{#1}} % short for mathscript
\newcommand{\mbf}[1]{\mathbf{#1}} % short for math bold face
\newcommand{\mt}[1]{\text{#1}}

\def\oa{\overline{\alpha}}

\begin{document}

\title{Monoid Embeddings of Symmetric Varieties}

\author[1]{Mahir Bilen Can}
\author[2]{Roger Howe}
\author[3]{Lex Renner}
\affil[1]{{\small Tulane University, New Orleans; mahirbilencan@gmail.com}}
\affil[2]{{\small Yale University, New Haven; roger.howe@yale.edu}}
\affil[3]{{\small Western University; lex@uwo.edu}}
\normalsize

\date{}
\maketitle

\begin{abstract}
We determine when an antiinvolution on an 
adjoint semisimple linear algebraic group extends 
to an antiinvolution on a $J$-irreducible monoid.
Using this information, we study a special class 
of compactifications of symmetric varieties. 
Extending the work of Springer on involutions, we describe the 
parametrizing sets of Borel orbits in these special 
embeddings. 
\\ \\
{\em Key words: Borel orbits, reductive monoids, symmetric varieties}
\\ \\
{\em AMS-MSC:} 14M17, 20G05, 20M32
\end{abstract}

\section{Introduction}

Let $G$ be a complex reductive  
algebraic group and let $\theta : G\rightarrow G$ be 
an algebraic group automorphism such that $\theta^2 = id$. 
The fixed subgroup $H:=\{ g\in G:\ \theta(g) = g\}$ 
is called the symmetric subgroup associated with $\theta$
and the corresponding quotient $G/H$ is called 
a symmetric variety. Let $B$ be a Borel subgroup of $G$. 
From the works of Richardson and Springer in~\cite{RS90} 
and Helminck in~\cite{Helminck,Helminck96}, we know that 
there is a close relationship between the set of 
$B$-orbits in $G/H$ and the set of 
involutions in the Weyl group of $G$. In particular,
we know that the number of $B$-orbits in $G/H$ is finite,~\cite{Matsuki79,Springer83}.
A purpose of our paper is to show that 
the sets of $B$-orbits in certain ``monoid embeddings'' of 
the symmetric varieties are closely related to the
sets of involutions in certain finite inverse semigroups. 
We proceed to explain what we mean by a monoid embedding.

A {\em reductive monoid} is a linear algebraic monoid 
whose group of units is a reductive algebraic group. 
Let $M$ be a reductive monoid and let $\theta_{an}$ 
be an antiinvolution, that is to say, 
$\theta_{an}: M\rightarrow M$
is an automorphism of $M$ such that $\theta_{an}^2 = id$ 
and for every $m_1,m_2\in M$ we have 
$\theta_{an}(m_1 m_2) = \theta_{an}(m_2)\theta_{an}(m_1)$.
If $G$ denotes the group of units of $M$, then 
we denote the restriction of $\theta_{an}$ to $G$ by the same notation. 
Then the morphism $\theta: G\rightarrow G$ defined 
by $\theta(g):= \theta_{an}(g)^{-1}$ ($g\in G$)
is an involutory algebraic group automorphism. As before, 
let us denote by $H$ the fixed subgroup of $\theta$.
The morphic map 
\begin{align}
\tau : & M\longrightarrow M \notag \\
& m \longmapsto m \theta_{an}(m) \label{A:tau on M}
\end{align}
restricts to give a morphism on $G$ and we denote 
this restriction by $\tau$ as well. The image of $\tau$
on $G$ is denoted by $P$. Note that $P$ is a closed subvariety of $G$, 
and furthermore, it is isomorphic to $G/H$ as a 
variety (see~\cite[Lemma 2.4]{Richardson82}).

The {\em twisted (conjugation) action} of $G$ on $M$, denoted by $*$, is 
defined as follows:
\begin{align*}
g * m = g m \theta_{an} (g) = g m \theta(g)^{-1}\qquad (g\in G, m\in M).
\end{align*} 
It is easy to check that 
$P$ is stable under the twisted action. 
In fact, $P=G * 1_{G}$. 
It is not hard to see also that the 
set of $B*$-orbits in $P$ is in bijection with the 
set of $B$-orbits in $G/H$. 
We call the Zariski closure of $P$ in $M$
the monoid embedding of $P$. 
Since $P$ is isomorphic to $G/H$ and since 
$\overline{P}$ is $G*$-stable, we view 
the monoid embedding of $P$ as an equivariant
embedding of $G/H$ in the reductive monoid $M$.

An important auxiliary variety for our purposes is the fixed subvariety 
$Q$ defined by 
\begin{align*}
Q:=  \{ g\in G:\ \theta_{an}(g) = g \}. 
\end{align*}
It is easy to check that $Q$ is closed in $G$, $P\subset Q$, and that $Q$ is $G*$-stable. 
We know from Springer's work~\cite{Springer83} that 
if $B$ is a $\theta$-stable Borel subgroup of $G$,
then $Q$ has only finitely many $B*$-orbits. 
As in the case of $P$, the parametrizing set of 
$B*$-orbits in $Q$ is closely related to the set of 
involutions in the Weyl group 
$W:=N_G(T)/T$, where $T$ is a maximal torus contained in $B$
and $N_G(T)$ is the normalizer of $T$ in $G$.
(Here, by an involution in $W$ we mean an element $\sigma\in W$ 
such that $\sigma^2 =id$.)

Let $M_Q$ denote the following (closed) subvariety of $M$:
\begin{align}\label{A: definition of M_P}
M_Q:=\{ m \in M :\ \theta_{an} (m) = m \}.
\end{align}
Clearly, $\overline{P}\subseteq \overline{Q} \subseteq M_Q$
and $G$ acts on the sets $P,Q,\overline{P},\overline{Q}$, 
and on $M_Q$ by the same formula 
$g *m = g m \theta_{an}(g)$. 
Our first main result is about the parametrizing sets of 
$B*$-orbits in the embeddings of $P$ and $Q$ in $M$.

\begin{Theorem}\label{T:Corollary}
Let $M$ be a normal reductive monoid with unit group $G$, $\theta_{an}$ be an antiinvolution on $M$,
and let $\theta$ denote the involutive automorphism on $G$ that is defined 
by $\theta(g) = \theta_{an}(g)^{-1}$ for $g\in G$. We fix a pair $(T,B)$ of $\theta$-stable 
maximal torus and a Borel subgroup in $G$,
and we let $\overline{N}$ denote the Zariski closure in $M$ of the normalizer of $T$ in $G$.
In this case, the following sets are finite and they are in bijection with each other;
\begin{enumerate}
\item $B*$-orbits in $\overline{Q}$ (respectively, $B*$-orbits in $\overline{P}$), 
\item $T\times H$-orbits in $\tau^{-1} (\overline{N} \cap \overline{Q})$ (respectively, 
$T\times H$-orbits in $\tau^{-1} (\overline{N} \cap \overline{P})$).
\end{enumerate}
\end{Theorem}

The {\em Renner monoid} of $M$, defined by $R=\overline{N}/T$, 
is a generalization of the Weyl group of $G$, see~\cite{Renner86}. 
It is a finite inverse semigroup and $W$ is its group of invertible elements.
As a consequence of Theorem~\ref{T:Corollary}, 
we will show in the sequel that a certain subset 
of $R$ can be used for studying $B*$-orbits in $\overline{P}$. 
In some special cases this subset of $R$ give a parametrization of 
the full set of $B*$-orbits; see the examples in Section~\ref{S:Borel}.

\vspace{.5cm}

With hindsight, our first main result raises the question 
of finding antiinvolutions on reductive monoids. 
To answer this question, in our second main result, 
we focus on a particular subclass of reductive monoids.
A {\em semisimple monoid} is a reductive monoid 
which is normal, has a one dimensional center and 
a zero element. Interesting examples of such monoids 
include the cones over certain representations 
of semisimple groups.

Let $G_0$ be a semisimple algebraic group of adjoint type and let 
$\rho :G\rightarrow \mt{GL}(V)$ be 
a finite dimensional irreducible rational representation of $G_0$.
Let $Z_V$ denote the cone over $\rho(G)$ in $\mt{End}(V)$.
Then $Z_V$ has the structure of a reductive monoid. 
Let us mention that $Z_V$ is known to be normal 
as an algebraic variety if the representation $(\rho,V)$ is 
a minuscule representation in the sense of~\cite[Theorem 3.1]{DeConcini}.
(We will review De Concini's theorem in the preliminaries section.)
In the following result we denote by $G$
the reductive group of units in $Z_V$.

\begin{Theorem}\label{T:extension of involution}
Let $G_0$ be a complex semisimple algebraic group of adjoint type 
and let $\theta_0 \in \mt{Aut}(G_0)$ be an involutory algebraic 
group automorphism. 
If $(\rho,V)$ is a minuscule representation of $G_0$
with the highest weight $\omega$ such that 
$\theta_0^* \omega = -\omega$ and $Z_V$ is normal, 
then there exists a unique morphism 
$\theta_{an}: Z_V \rightarrow Z_V$ such that 
\begin{enumerate}
\item $\theta_{an} (xy) = \theta_{an}(y) \theta_{an}(x)$ for all $x,y\in Z_V$;
\item $\theta_{an}^2$ is the identity map on $Z_V$;
\item $\theta_{an} (g) = \theta(g)^{-1}$ for all $g\in G$, where $\theta$
is the unique extension of $\theta_0$ to $G$.
\end{enumerate}
\end{Theorem}

We conclude our introduction by giving a brief overview of our article. 
In Section~\ref{S:Preliminaries}, we set our notation and 
review some facts from the theory of linear algebraic monoids
and the representation theory of reductive algebraic groups. 
In Section~\ref{S:Special}, we prove our second main result, 
Theorem~\ref{T:extension of involution}. 
In Section~\ref{S:Borel}, we characterize the parametrizing 
sets of Borel orbits in $\overline{P}$. 
In particular we prove our Theorem~\ref{T:Corollary} in Section~\ref{S:Borel}.
Finally, we close our paper with some remarks in Section~\ref{S:Final}.

%	\noindent \textbf{Acknowledgement.}
%	The authors thank the anonymous referee for her/his comments and suggestions which improved the 
%	quality of the paper drastically.  
%	The first author is partially supported by the Louisiana Board of Regents 
%	Research and Development Grant 549941C1 and by the N.S.A. Grant H98230-14-1-0142.

\section{Preliminaries}\label{S:Preliminaries}

Unless otherwise mentioned, all reductive groups are assumed to be connected
and all semigroups are defined over $\C$. The representations we consider here
are all rational and finite dimensional. 

The general linear group of invertible $n\times n$ matrices is denoted by $\mt{GL}_n$
and the monoid of $n\times n$ matrices is denoted by $\mt{Mat}_n$. 
The Lie algebra of a linear algebraic group $G$ is denoted by $\mt{Lie}(G)$.

Let $G$ be a reductive algebraic group,
$T$ be a maximal torus in $G$, and let 
$B$ be a Borel subgroup containing $T$. 
We use $X(T)$ to denote 
the character group of $T$, and we will
use $E$ to denote the real vector space 
$X(T) \otimes_\Z \R$. In addition, we 
fix the following notation:
\begin{align*}
\varPhi \subset E&: \text{ the set of weights of the adjoint representation};\\
\varDelta \subset \varPhi &: \text{ the set of simple roots determined by } (B,T);\\
\varPhi^+  \subset \varPhi &: \text{ the set of positive roots determined by } \varDelta;\\
\varLambda_r \subset X(T)&: \text{ the root lattice generated by $\varDelta$}.
\end{align*}

If $\alpha$ is a root from $\varPhi$, then the associated coroot,
$2\alpha / (\alpha, \alpha)$, is denoted by $\check{\alpha}$. Suppose that 
$\alpha_1,\dots, \alpha_n$ is the list of simple roots from $\varDelta$.
The set of {\em fundamental weights}, $\{\omega_1,\dots, \omega_n\}$ 
is the dual of the coroot basis $\{\check{\alpha_1},\dots,\check{\alpha_n}\}$ 
for the dual vector space $\mt{Lie}(T)^*$. 

Irreducible representations of $G$ are parametrized by the semigroup 
of dominant weights (with respect to $T$). A dominant weight $\lambda$
is called minuscule if $\langle \lambda , \check{\alpha} \rangle  \leq 1$ 
for all positive coroots $\check{\alpha}$.

The dominance partial order on weights is defined by
$\mu \preceq \lambda$ if and only if $\lambda-\mu$ is a positive 
linear combination of positive roots.

If $\lambda$ is a dominant weight, then we denote by 
$\it{\Sigma} (\lambda)$ the set of dominant weights $\mu$
such that $\mu \preceq \lambda$. The set $\it{\Sigma}(\lambda)$ is 
finite and it is called the saturation of $\lambda$.

\subsection{Reductive monoids.}\label{SS:background}

The purpose of this section is to introduce the notation of a reductive monoid. 
For details, see~\cite{Renner05,Putcha88}.
For a more recent exposition of the basic ideas
behind algebraic monoids we recommend Brion's article~\cite{Brion14} and
for the combinatorics of Renner monoids, we recommend~\cite{CaoLiLi14}.

Let $M$ be a linear algebraic monoid with 
the group of invertible elements $G$. 
The set of idempotents in $M$ is denoted
by $E(M)$. 
If $G$ is a reductive group and $M$ is an 
irreducible algebraic variety, then $M$ 
is called a reductive monoid. Note that 
there is no normality assumption on $M$.

Let $T$ be a maximal torus in $G$
and let $B$ be a Borel subgroup containing $T$. 
Clearly, $\overline{T}$ is a reductive 
and commutative submonoid of $M$.
As before, we denote by $R$ (resp. by $W$), the 
Renner monoid $\overline{N_G(T)}/T$ (resp. the Weyl 
group $N_G(T)/T$) of $M$ (resp. of $G$).

%The{\em  Bruhat-Chevalley order} on $W$ is defined by 
%\begin{equation*}
%x \leq y\ \hspace{.51cm} \text{if and only if}\hspace{.51cm}  BxB \subseteq \overline{ByB}.
%\end{equation*}
The ``generalized'' Bruhat-Chevalley order on $R$ is defined by 
\begin{equation}\label{E:BRordering}
\sigma \leq \tau\ \hspace{.51cm} \mbox{if and only if}\hspace{.51cm} 
B\sigma B \subseteq \overline{B\tau B}
\end{equation}
where $\tau$ and $\sigma$ are from $R$
and the bar on $B\tau B$ stands for the Zariski closure in $M$.

There is a canonical partial order $\leq$ 
on the set of idempotents $E(\overline{T})$ of $\overline{T}$
defined by  
\begin{equation}\label{E:crossorder}
e\leq f \hspace{.51cm} \mbox{if and only if} \hspace{.51cm}  ef=e=fe.
\end{equation}
Notice that $E(\overline{T})$ is invariant under the conjugation 
action of the Weyl group $W$. 
A subset $\Lambda \subseteq E(\overline{T})$ is called a cross-section lattice (or, a {\em Putcha lattice})
if $\Lambda$ is a set of representatives for the $W$-orbits on $E(\overline{T})$ and the bijection 
$\Lambda \rightarrow G \backslash M / G$ defined by 
$e \mapsto GeG$ is order preserving.
There is a close relationship between cross-section 
lattices and Borel subgroups. 
The {\em right centralizer of $\Lambda$ in $G$}, denoted by $C_G^r(\Lambda)$, is the subgroup 
$$
C_G^r(\Lambda) = \{ g\in G:\ ge = ege \text{ for all } e\in \Lambda\}.
$$	
%	It turns out that given a cross-section lattice $\Lambda$, there exists a unique Borel subgroup $B=B(\Lambda)$ 
%	such that $B= C_G^r(\Lambda)$. Conversely, any cross-section
%	lattice is of the form $\Lambda=\Lambda(B) = \{ e\in E(\overline{T}):\ Be=eBe\}$
%	for a unique Borel subgroup $B$ containing $T$.  
Assuming that $M$ has a zero, for all Borel subgroups of $G$ containing $T$ 
the set $\Lambda(B) = \{ e\in E(\overline{T}):\ Be=eBe\}$ is a cross-section lattice 
with $B= C_G^r(\Lambda)$, and for any cross-section lattice $\Lambda$, the 
right centralizer $C_G^r(\Lambda)$ is a Borel subgroup containing $T$ with 
$\Lambda = \Lambda (C_G^r(\Lambda))$. See~\cite[Theorem 9.10]{Putcha88}.
	
\vspace{.25cm}

The decomposition $M= \bigsqcup_{e\in \Lambda} GeG$ into 
$G\times G$ orbits has a finite counterpart; $R = \bigsqcup_{e\in \Lambda} WeW$.
Moreover, the partial order (\ref{E:crossorder}) on 
$\Lambda$ agrees with the order induced from 
Bruhat-Chevalley order (\ref{E:BRordering}).

\vspace{.25cm}

$E(\overline{T})$ is a relatively complemented lattice, anti-isomorphic 
to a face lattice of a convex polytope. For $\overline{T}$ contained in a 
$J$-irreducible  monoid, the associated polytope is described explicitly in 
Section~\ref{SS:Jirreducible}.
Let $\Lambda$ be a cross section lattice in $E(\overline{T})$.  
The Weyl group of $T$ (relative to $B= C_G^r(\Lambda)$) acts on $E(\overline{T})$, 
and furthermore
\begin{equation*}
E(\overline{T})= \bigsqcup_{w\in W} w \Lambda w^{-1}.
\end{equation*}

Let $S$ be a semigroup and let $M=S^1$ be the monoid obtained from 
$S$ by adding a unit element if it is not already present. 
Let $a,b\in M$. The following are four of the five equivalence relations 
which are collectively known as Green's relations.
They are of utmost importance for semigroup theory. 
\begin{enumerate}
\item $a\ \ms{L}\ b$ if $Ma=Mb$.
\item $a\ \ms{R}\ b$ if $aM=bM$.
\item $a\ \ms{J}\ b$ if $MaM=MbM$.
\item $a\ \ms{H}\ b$ if $a\ \ms{L}\ b$ and $a\ \ms{R}\ b$.
\end{enumerate}
It turns out that the unit group $G$ of a reductive monoid $M$ is big in the sense that 
$a\ \ms{L}\ b$ if $Ga=Gb$, $a\ \ms{R}\ b$ if $aG=bG$, and $a\ \ms{J}\ b$ if $GaG=GbG$
(see~\cite[Proposition 6.1]{Putcha88}).
Furthermore, a cross section lattice is a representative for the set of $\ms{J}$-classes in $M$.
A reductive monoid $M$ is called $J$-irreducible if $M$ has a unique, nonzero, minimal $G\times G$-orbit. 

\vspace{.5cm}

We continue with the assumption that $M$ is a reductive group with unit group $G$. 
Among the important submonoids of $M$ are those of the form $eMe$ ($e\in E(\overline{T})$). 
Let $C_G(e)$ denote the centralizer of $e$ in $G$. If $e$ is from the cross-section
lattice $\Lambda$, then $eC_G(e)$ is the unit group of $eMe$.
In the sequel, we will need the following fact, also.
\begin{Lemma}\label{L:H-class}
Let $e\in E(\overline{T})$ be an idempotent 
and let $H$ denote its $\ms{H}$-class, that is $H= e C_G(e)$. 
Let $B$ be a Borel subgroup of $G$ containing $T$, hence $e\in E(\overline{B})$. 
In this case, $C_B(e)$ and $eBe = e C_B(e)$, respectively, are 
Borel subgroups of $C_G(e)$ and $H$. 
\end{Lemma}
For the proofs of the facts that are stated in the previous paragraph
as well as for the proof of the lemma, see~\cite[Corollary 7.2]{Putcha88}.

\subsection{$J$-irreducible  monoids.}\label{SS:Jirreducible}

Let $G_0$ denote a semisimple linear algebraic 
group of adjoint type, $T_0$ be a maximal torus
in $G_0$. 
If $(\rho_0,V)$ is a representation of $G_0$, 
then the group $\C^*\cdot \rho_0(G_0)$, 
which we denote by $G$, is reductive. 
If $\rho_0$ is faithful, then up to isomorphism 
$T_0$ and $\rho_0(T_0)$ 
differ by a finite set of central elements.
In this case, when there is no danger of confusion, 
we will denote the image $\rho_0(T_0)$ by $T_0$.

Let $T\subseteq G$ be a maximal torus containing $T_0$, 
and let $\mbf{T}\subset \mt{GL}(V)$ 
denote an $n$-dimensional maximal torus containing $T$. 
(Here, $n=\dim V$.)
Accordingly, we have a nested sequence of Euclidean spaces: 
$$
E_0=X(T_0)\otimes_\Z \R \subset E=X(T)\otimes_\Z \R \subset \mbf{E}=X(\mbf{T})\otimes_\Z \R.
$$
Note that $\dim E = \dim E_0 +1$.

If $\varepsilon_i$ (for $i=1,\dots, n$) denotes 
the standard $i$-th coordinate function on $\mbf{T}$, 
then $\{ \varepsilon_1,\dots, \varepsilon_n \}$ is a basis for $\mbf{E}$, and $E$ is spanned 
by the restrictions $\varepsilon_i |_T$, $i=1,\dots, n$. 
Let $\chi \in X(T)$ denote the restriction of the character whose $n$-th power 
is the determinant on $\mt{GL}(V)$. Stated differently in additive notation, $\chi$ is the restriction to 
$T$ of the rational character $\frac{1}{n}(\varepsilon_1+\cdots + \varepsilon_n)$. 
We denote $\varepsilon_i |_T$ by $\chi_i$ and set 
$$
\widetilde{\chi}_i :=\chi_i - \chi\ \text{ for } i=1,\dots, n.
$$

If $K$ is an arbitrary group, then the center of $K$ is customarily denoted by $Z(K)$. 
In our case, since $Z(G)= \C^*\cdot Z(G_0)$, the character group of $Z(G)$ is generated 
by one element, which is $\chi$. Thus, $E = \R \chi \oplus E_0$. In fact, $\chi$ vanishes on $T_0$. 
It follows from these observations that 
\begin{enumerate}
\item $\{ \widetilde{\chi}_1,\dots, \widetilde{\chi}_{n-1},\chi\}$ spans $E$;
\item if $x\in T$ lies in $T_0\subset T$, then $\widetilde{\chi}_i|_{T_0} (x) = \chi_i (x)$;
\item $\{\widetilde{\chi}_1,\dots, \widetilde{\chi}_{n-1}\}$ spans $E_0$.
\end{enumerate}

In this paper, we are interested in the $J$-irreducible  monoids that come from a faithful representation.
\begin{Definition}\label{D:$J$-irreducible }
Let $(\rho_0,V)$ be a faithful, irreducible representation of $G_0$. 
The {\em $J$-irreducible  monoid associated with $(\rho_0,V)$} is the affine variety 
$\overline{\C^*\cdot \rho_0(G_0)}$ together with its monoid structure induced from $\mt{End}(V)$. 
\end{Definition}
\begin{Remark}
The definition of a $J$-irreducible  monoid which is given at the end of Section~\ref{SS:background} 
agrees with Definition~\ref{D:$J$-irreducible }, see Lemma 7.8 of \cite{Renner05}.
\end{Remark}

The representation $(\rho_0,V)$ of $G_0$ gives a representations for $G$ by the following action:
\begin{align}\label{A:family of reps}
zg \cdot v = z \rho_0(g)(v) \in V,
\end{align} 
where $z\in \C^*$, $g\in G_0$, and $v\in V$. 
Another such simple but useful observation is that, since $G_0 \stackrel{\rho_0}{\cong} (G,G)$,
the Weyl group $W$ of $(G_0,T_0)$ is isomorphic to that of the pair $(G,T)$.

\begin{Lemma}\label{L:convex hull}[Proposition 3.5 in \cite{Renner85}]
Let $(\rho_0,V)$ be an irreducible representation of $G_0$. 
If $\lambda$ is the $T_0$-highest weight of $(\rho_0,V)$ and $\mc{P}$ denotes the convex hull of 
$\{ w\cdot (\chi+ \lambda):\ w\in W \} = \{\chi+w \cdot \lambda:\ w\in W \}$, then the set of weights of $T$ 
with respect to~(\ref{A:family of reps}) is contained in $\mc{P}$.
\end{Lemma}

Next, we briefly review a result of De Concini on the the normality of 
the $J$-irreducible monoids. Let $\lambda$ be a dominant weight for $G_0$ and let 
$(\rho_0,V)$ denote the corresponding irreducible representation of $G_0$. 
We define $(\eta,W_\lambda)$ as the following sum of irreducible representations 
of $G_0$:
$$
W_\lambda := \bigoplus_{\mu \in \it{\Sigma}(\lambda) } V(\mu).
$$
Here $V(\mu)$ stands for the irreducible representation of $G_0$ 
with highest weight $\mu$. Finally, we set 
$$
Z_\lambda:= Z_V \qquad \text{and}\qquad \mc{Z}_\lambda := Z_{W_\lambda},
$$
where $Z_{W_\lambda}$ is the cone over $\eta(W_\lambda)$ in $\mt{End}(W_\lambda)$.

\begin{Theorem}(De Concini~\cite[Theorem 3.1]{DeConcini})\label{T:DeConcinis}
1) $\mc{Z}_\lambda$ is a normal variety with rational singularities.
2) If $V$ is a $G_0$-module of highest weight $\lambda$, then $\mc{Z}_\lambda$ is 
the normalization of $Z_V$ and it is equal to $Z_V$ if and only if $W_\lambda$ is a 
subrepresentation of $V$. In particular, $\mc{Z}_\lambda$ is a normalization of $Z_V$
and it is equal to $Z_V$ if and only if $\lambda$ is minuscule. 
\end{Theorem}

\section{A proof of Theorem~\ref{T:extension of involution}}\label{S:Special}

Let $G_0$ be a semisimple algebraic group 
of adjoint type, $\theta_0$ be an involutory 
linear algebraic group automorphism of $G_0$. 
Let $(T_0,B_0)$ be a $\theta_0$-stable 
pair of a maximal torus $T_0$ and a Borel subgroup
$B$ such that $T_0\subset B_0$. 
The {\em isotropic subtorus} $T_0'$, and the 
{\em anisotropic subtorus} $T_1'$ are defined by
\begin{align*}
T_0' = \{ t\in T_0:\ \theta(t) = t\},\qquad
T_1' = \{ t\in T_0:\ \theta(t) = t^{-1} \}.
\end{align*} 
The multiplication map $T_1' \times T_0' \rightarrow T_0$ 
is an {\em isogeny}, that is to say 
a surjective homomorphism with a finite kernel.

Among all $\theta$-stable maximal tori, 
we work with the one for which the dimension 
$l:=\dim T_1'$ is maximal. The integer $l$ is 
called the rank of the symmetric variety $G_0/G_0^\theta$.
($G_0^\theta$ is the fixed subgroup of $\theta$.)

Let $\varPhi$ denote the set of roots of $G_0$ relative to $T_0$. 
Passing to the Lie algebra setting by differentiation, we view $\varPhi$ as a subset of the 
dual vector space $\mt{Lie}(T_0)^*$ of the Lie algebra of $T_0$.	
Since $\theta$ is an automorphism of $T_0$, it induces a linear map 
$$
\theta^* : \ \mt{Lie}(T_0)^* \rightarrow \mt{Lie}(T_0)^*,
$$
which, in turn induces an involution on $\varPhi$. 
Define 
\begin{align*}
\varPhi_0 &= \{ \alpha \in \varPhi:\ \theta^*(\alpha) = \alpha \},\\
\varPhi_1 &= \varPhi - \varPhi_0.
\end{align*} 
%	The dimension of the symmetric space $G/G^\theta$ is related to these invariants by the formula 
%	$\dim G/G^\theta = l + \frac{| \varPhi_1 |}{2}$ (see Section 1.3 of~\cite{DeConciniProcesi83}).

\begin{Lemma}[Lemma 1.2, \cite{DeConciniProcesi83}]\label{DeConciniProcesi83 Lemma 1.2}
There exists a system of positive roots $\varPhi^+ \subseteq \varPhi$ such that 
$\theta^*(\alpha) \in \varPhi- \varPhi^+$ for all $\alpha \in \varPhi^+ \cap \varPhi_1$.
\end{Lemma}

We fix a set of positive roots $\varPhi^+$ as in Lemma \ref{DeConciniProcesi83 Lemma 1.2}.
Let $\varDelta$ denote the associated set of simple roots and set 
\begin{align*}
\varDelta_0 &= \varPhi_0 \cap \varDelta,\\
\varDelta_1 &= \varPhi_1 \cap \varDelta. 
\end{align*} 
Observe that $|\varDelta_1| \geq \dim T_1'=l$. 	
It turns out that there exists an ordering $\{\alpha_1,\dots, \alpha_j\}$ of the elements of $\varDelta_1$ 
such that the differences $\alpha_ i - \theta^*(\alpha_i)$ are mutually distinct for $i=1,\dots,l$, and 
for each $i\in \{ l+1,\dots,j\}$, there exists an index $s\in \{1,\dots, l\}$ such that 
$\alpha_ i - \theta^*(\alpha_i) =  \alpha_ s - \theta^*(\alpha_s)$.
See~\cite[Section 1.4]{DeConciniProcesi83}.
A {\em restricted simple root} $\overline{\alpha}$ is a weight of the form  
\begin{align*}
\oa= \frac{\alpha_i - \theta^*( \alpha_i)}{2} \text{  for some } i\in \{1,\dots, l\}.
\end{align*}
In this case, we denote $\oa$ by $\oa_i$, and denote by 
$\overline{\varDelta_1}=\{\oa_1,\dots,\oa_l\}$ the set of all restricted simple roots.

Suppose now that $\varDelta_0 = \{\beta_1,\dots, \beta_k\}$. 
In accordance with the partitioning $\varDelta= \varDelta_0 \sqcup \varDelta_1$, 
we divide the set of fundamental weights of $\varDelta$ into two disjoint subsets
$\{\omega_1,\dots, \omega_j\} \sqcup \{ \zeta_1,\dots, \zeta_k\}$ so that for each 
$i\in \{1,\dots, j\}$ the following equalities hold true: 
\begin{center}
$(\omega_i, \beta_s^\lor) = 0$ for $s=1,\dots,k$, and $(\omega_i , \alpha_r^\lor) = \delta_{i,r}$ 
for $r=1,\dots, j$.
\end{center}
Similarly for $\zeta_i$'s.
As it is shown in \cite{DeConciniProcesi83} (in the pages 5 and 6), $\theta^*$ induces an involution 
$\widetilde{\theta}$ on the indices $\{1,\dots, j\}$ such that 
$\theta^*( \omega_i ) = - \omega_{\widetilde{\theta}(i)}$.
Thus, we arrive at a crucial definition for our purposes: 
\vspace{.25cm}

\begin{Definition}\label{D:special weight}
A dominant weight $\lambda$ of $G_0$ is called special (or, $\theta$-special), 
if $\theta^*(\lambda) = - \lambda$. 
If $(\rho,V)$ is an irreducible representation with a $\theta$-special highest weight, then 
we call $\rho$ a $\theta$-special representation of $G_0$. 
\end{Definition}

Now, let $\theta_0: G_0\rightarrow G_0$ be an involutory automorphism on $G_0$. 
We choose a $\theta_0$-stable maximal torus $T_0$ in $G_0$. 
Let $\lambda$ be a special, dominant weight with the corresponding irreducible representation $(\rho_0,V)$.
Assume also that $(\rho_0,V)$ is faithful. 
As before, we define the reductive group $G$ by setting $G=\C^*\cdot \rho_0(G_0) \subset \mt{GL}(V)$.
We claim that there exists an ``extension'' $\theta :G\rightarrow G$ of $\theta_0$ to $G$. 
To this end we define 
\begin{align}\label{A:extension of theta}
\theta (c\rho_0(g))= c^{-1} \rho_0(\theta_0(g)), \qquad g\in G_0,\ c\in \C^*.
\end{align}
To prove that $\theta$ is well defined, suppose $g,g'\in G_0$ and $c,c'\in \C^*$ are such that 
$c\rho_0(g) = c'\rho_0(g')$. Let $\alpha \in \C^*$ denote $cc'^{-1}$. 
Then $\rho_0(g^{-1} g') = \alpha 1_{\mt{GL}(V)} \in \mt{GL}(V)$.
But $G_0$ is of adjoint type, $\rho_0$ is faithful, and $\alpha$ is a central element in $\mt{GL}(V)$. 
Therefore, $\alpha = id$, hence $g=g'$ and $c=c'$. 
Finally, note that 
\begin{align*}
\theta(\theta (c\rho_0(g))) &= \theta(c^{-1} \rho_0(\theta_0(g)))\\
&=c \rho_0(\theta_0(\theta_0(g)))\\
&=c \rho_0(g)\ \text{ for all } c\in \C^* \text{ and } g\in G_0.
\end{align*}

The {\em antiinvolution} corresponding to $\theta$, by definition, is 
the composition $\theta_{an}:= \theta \circ \iota$ of $\theta$ with the 
``inverting'' morphism $\iota:\ g\mapsto g^{-1}$. 
The map induced by $\theta_{an}$ on the character group $X(T)$ is denoted $\theta_{an}^*$.  
Then $\theta^*$ and $\theta_{an}^*$ are related to each other by 
$$
\theta_{an}^*(\chi) = - \theta^*(\chi)\ \text{ for } \chi \in X(T).
$$  
In particular, if $\theta^* ( \lambda ) = - \lambda$, then $\theta_{an}^* (\lambda) = \lambda$.

We are ready to prove Theorem~\ref{T:extension of involution}. 
Let us paraphrase it for completeness:
If $M$ is a normal $J$-irreducible monoid that is obtained from 
a $\theta$-special minuscule representation of $G$, 
then there exists a unique morphism $\theta_{an}:M\rightarrow M$ such that 
\begin{enumerate}
\item $\theta_{an} (xy) = \theta_{an}(y) \theta_{an}(x)$ for all $x,y\in M$;
\item $\theta_{an}^2$ is the identity map on $M$;
\item $\theta_{an} (g) = \theta(g)^{-1}$ for all $g\in G$, where $\theta$ is the 
involution that is extended from $\theta_0$ on $G_0$.
\end{enumerate}

\begin{proof}[Proof of Theorem~\ref{T:extension of involution}.]
Since $\theta_{an}$ agrees with $\theta$ (after composing with $\iota$, of course) on $G$, 
the uniqueness is clear. 
We are going to show that $\theta_{an}$ extends to whole $J$-irreducible monoid $M:=Z_V$ associated 
to an irreducible representation $(\rho_0,V)$ of $G_0$ with the highest weight $\lambda$. Let $\rho$
denote the representation of $G$ as defined in (\ref{A:family of reps}).

First, we note that, by Theorem~\ref{T:DeConcinis}, $M$ is a normal reductive monoid. 
Let $T_0$ denote the maximal torus of $G_0$ such that $T= \C^*\cdot \rho_0(T_0)$, 
and let $\langle \Pi(\rho) \rangle$ denote the submonoid of $X(T)$ generated by the weights $\Pi(\rho)$ of $T$. 
The coordinate ring of the affine torus embedding $\overline{T}$ is equal to the 
monoid-ring $R= \C [ \langle \Pi(\rho) \rangle ]$ (see [Lemma 3.2,~\cite{Renner85}]). 
Therefore, $\overline{T} = \mt{Spec}(R)$.
On the other hand, by Lemma~\ref{L:convex hull}, we know that $\Pi(\rho)$ is 
contained in the convex hull $\mc{P}$ of $W\cdot (\lambda+\chi)$, where $\chi$ is the $n$-th 
root of the determinant on $\mt{GL}(V)$.

Since $\lambda$ is special, by [Lemma 1.6,~\cite{DeConciniProcesi83}], there is a $G$-isomorphism 
$V^\theta \simeq V^*$, hence $\theta^*(\Pi(\rho)) = \Pi(\rho^*)=-\Pi(\rho)$ 
and it follows $\theta^*_{an}(\Pi(\rho))=\Pi(\rho)$.	
In particular, it induces an antiinvolution $\theta_{an}$ on $\overline{T} = \mt{Spec}(R)$.
Since $\theta \circ \iota = \theta_{an}$ on $T$, by the ``extension principle'' (see [Corollary 4.5,~\cite{Renner85}]), 
there exists a unique morphism $\theta_{an} : M\rightarrow M$, which agrees with $\theta \circ \iota$ on $G$, 
and agrees with $\theta_{an}$ on $\overline{T}$. 	
Since $\theta_{an}^2 =id$ on $G$, and since $G$ is dense on $M$, we see that $\theta_{an}^2= id$ on $M$.
Finally, since $(x,y) \mapsto \theta_{an}(xy)$ and $(x,y) \mapsto \theta_{an}(y) \theta_{an}(x)$ 
are morphisms from $M \times M$ into $M$ agreeing on the open dense set 
$G\times G$, they agree everywhere. 
	
\end{proof}

\section{A proof of Theorem~\ref{T:Corollary}}\label{S:Borel}

We start with providing the details of some useful facts which we briefly 
mentioned earlier in Section~\ref{SS:background}.

\begin{Lemma}[Generalized Bruhat-Chevalley decomposition, \cite{Renner86}]\label{BCR}
Let $M$ be a reductive monoid with the group of invertible elements $G$, and let $T\subseteq B$ 
be a maximal torus contained in a Borel subgroup. Let $\overline{N}$ denote the closure in $M$ 
of the normalizer $N= N_G(T)$ of $T$. If $m\in M$, then there exist $b_1,b_2 \in B$, and 
$\overline{n} \in \overline{N}$ such that $m= b_1 \overline{n} b_2$. 
This leads to the Bruhat-Chevalley decomposition of $M$:
\begin{align}\label{A:disjoint}
M= \bigcup_{\dot{\overline{n}} \in R} B \overline{n} B,
\end{align}
where the union is disjoint and $R=\overline{N}/T$ is the Renner monoid of $M$.
\end{Lemma}

Fix an element $\overline{n} \in \overline{N}$, and let $V_{\overline{n}} \subseteq U$ 
denote the subgroup $V_{\overline{n}} = \{ u \in U :\ u \overline{n} B \subseteq \overline{n} B \}$. 
Then $V_{\overline{n}}$ is closed and $T$-stable under conjugation.	
Therefore, there exists a complementary subgroup 
\begin{align}\label{complementary first}
U_{\overline{n},1} = \prod_{U_\alpha \nsubseteq V_{\overline{n}}} U_\alpha.
\end{align}
Complementary in this context means that the product morphism 
$U_{\overline{n},1} \times V_{\overline{n}} \rightarrow U$ is an isomorphism of algebraic groups.

In a similar manner, let $Z_{\overline{n}}\subseteq U$ denote the closed subgroup 
$Z_{\overline{n}}= \{ u \in U :\ \overline{n}T u = \overline{n}T\}$. 
Also in this case, $Z_{\overline{n}}$ is $T$-stable under conjugation;
let 
\begin{align}\label{complementary second}
U_{\overline{n},2} = \prod_{U_\alpha \nsubseteq Z_{\overline{n}}} U_\alpha
\end{align}
denote its complementary subgroup. 
The precise structure of the orbit $B\overline{n}B$ is exhibited in the next result:

\begin{Lemma}[Lemma 13.1, \cite{Renner05}]\label{BCR uniqueness}
The product morphism $U_{\overline{n},1} \times \overline{n} T \times U_{\overline{n},2} 
\rightarrow B \overline{n} B$ 
is an isomorphism of varieties.
\end{Lemma}

As a consequence of Lemmas~\ref{BCR} and~\ref{BCR uniqueness} 
we have the following important observation: \\

\noindent 
{\bf Uniqueness Criterion:} Given an element $m\in M$, 
there exist unique $u \in U_{\overline{n},1}, v \in U_{\overline{n},2}$, and $\overline{n} \in \overline{N}$ such that 
\begin{align}\label{uniqueness}
m = u \overline{n} v.
\end{align}

We continue with the assumption that $M$ is a reductive monoid with an antiinvolution 
$\theta_{an} : M\rightarrow M$. Let $H\subseteq G$ denote, as usual, the fixed subgroup $G^\theta$, 
where $\theta : G\rightarrow G$ is the involution $\iota \circ \theta_{an}$, where $\iota$ stands for the inverse map. 
Here, we are going to investigate the sets of $B*$-orbits in following varieties: 
\begin{itemize}
\item the Zariski closure $\overline{P}$ in $M$ of $P=\{ g\theta (g^{-1}) :\ g\in G \} \simeq G/H$;
\item the Zariski closure $\overline{Q}$ in $M$ of $Q=\{g \in G:\ \theta(g) = g^{-1} \}$;
\item and $M_Q := \{ x\in M:\ \theta_{an} (x)= x\}$.
\end{itemize}

Assume from now on that $T$ is a $\theta$-stable maximal torus of the $\theta$-stable
Borel subgroup $B\subseteq G$. Notice in this case that the corresponding unipotent subgroup $U \subset B$ 
has to be $\theta$-stable, as well. Moreover, since $T$ is $\theta$-stable, if $n$ is an element 
from the normalizer $N= N_G(T)$, then 
$\theta(n) t \theta(n)^{-1} = \theta(n) \theta(t') \theta(n^{-1})= \theta(n t' n^{-1}) \in T$ for some $t'\in T$. 
In other words, $\theta(N) = N$. It follows that the Zariski closure $\overline{N}$ is $\theta_{an}$-stable.

\begin{Proposition}\label{B orbits intersect with closure(N)}
Any $B*$-orbit in $M_Q$ contains an element of $\overline{N}$. 
\end{Proposition}

\begin{proof}

For $m\in M_Q$, as it is shown at the beginning of this section, there exist unique 
$u \in U_{\overline{n},1}, v \in U_{\overline{n},2}$, and $\overline{n}\in \overline{N}$ 
such that $m= u\overline{n} v$. Then 
$$
u \overline{n} v = m= \theta_{an} (m) = \theta_{an} (v) \theta_{an}(\overline{n}) \theta_{an} (u) 
= \theta(v)^{-1} \theta_{an}(\overline{n}) \theta(u)^{-1}.
$$
Since Bruhat-Chevalley decomposition (\ref{A:disjoint}) is a disjoint union, we see that 
$ \theta_{an} (\overline{n}) \in \overline{n}T$. Let $t\in T$ be such that  
$ \theta_{an} (\overline{n}) =\overline{n} t$.

Let $a$ denote $v\theta(u)$. 
It is clear that $\overline{n} a $ lies in the $B*$-orbit of $m$. 
Therefore,  we have 
\begin{align*}
\overline{n} a &= \theta_{an} (\overline{n} a )= \theta_{an} (a) \theta_{an} (\overline{n}) 
= \theta(a)^{-1} \overline{n}t \qquad \text{for some $t\in T$},
\end{align*}
or  
\begin{align}\label{normal}
\theta(a) \overline{n} a = \overline{n}t \qquad \text{for some $t\in T$}.
\end{align}	
Suppose $\overline{n} = e n$ for some $e\in E(\overline{T})$ and $n\in N$. 
By (\ref{normal}), we see that 
$$
\theta(a) e = e n t a^{-1} n^{-1}.
$$
In particular, we see the equality $\theta(a) e= e\theta(a) e$. 
Since $U$ is $\theta$-stable, we know that $\theta(a) \in U$, 
therefore, $\theta(a)e$ is an element of $eUe$. 
Notice that $nt = t' n$ for some $t'\in T$, so, $\theta(a) e = e t' n a^{-1} n^{-1}$.
At the same time, $n a^{-1} n^{-1}\in U$. 
In other words, $\theta(a) e = e t' u'$, where $t'\in T$, 
$u':=n a^{-1} n^{-1}\in U$, and we know
that $et'u'\in e U e$. Therefore, it is harmless to continue 
with $\theta(a) e = e u'$. 
%	It follows that the element $\theta(a)e= e \theta(a) e= e n a^{-1} n^{-1} e =e n a^{-1} n^{-1}$ 
%	is contained in the unit group of $e M e$. 
%	In particular, $e\theta(a)e$ lies in the unipotent radical of the Borel subgroup $eBe$ of the 
%	$\ms{H}$-class of $e$. See Corollary 7.2 (ii) \cite{Putcha88}.

Since square roots exists in unipotent groups, we see that 
$(e\theta(a) e )^{1/2} = e \theta(a)^{1/2} e = ( e n a^{-1} n^{-1} e)^{1/2} 
= e (n a^{-1} n^{-1})^{1/2}e = e ( n a^{-1/2} n^{-1}) e$.

The unit group of $eMe$ is $eC_G(e)$ and $eC_B(e)=eBe$ is a Borel subgroup of $eC_G(e)$ (see Lemma~\ref{L:H-class}).
Since $eUe\subseteq eC_B(e)$, we have that $e \theta(a)^{1/2} e = \theta(a)^{1/2}e$ and that $e ( n a^{-1/2} n^{-1}) e = e n a^{-1/2} n^{-1}$.
Now, on one hand we have $\theta(a)^{1/2}e= e n a^{-1/2} n^{-1}$, or equivalently $\theta(a)^{1/2} \overline{n} a^{1/2} = \overline{n}$.
On the other hand, $ \theta(a)^{1/2} \overline{n} a^{1/2} = \theta(a)^{1/2} * (\overline{n} a)$. 
Therefore, $\overline{n}$ is contained in the $B*$-orbit of $m$.
\end{proof}

\begin{Remark}\label{R:it fixes}
Recall that $\tau : M\rightarrow M$ is defined by $\tau(x) = x \theta_{an}(x)$. The image of $\tau$ is contained in $M_Q$. 
\end{Remark}
\begin{proof}
If $m\in M$, then 
$$
\theta_{an}(\tau(m)) = \theta_{an} ( m \theta_{an}(m)) = \theta_{an}(\theta_{an}(m)) \theta_{an}(m)=m\theta_{an}(m)= \tau(m).
$$
	
\end{proof}

%	Therefore, if $A\subseteq M$ is $\theta_{an}$-stable, then so are $\tau(A)$ and $\tau^{-1}(A)$. 
%	In particular, $\tau^{-1}(\overline{N})$ is $\theta_{an}$-stable. 

\begin{Lemma}\label{L:it acts}
$T\times H$ acts on $\tau^{-1}(\overline{N})$ by $(t,h) \cdot m = t m h^{-1}$. 
\end{Lemma}
	
\begin{proof}	
It suffices to check that for all $t\in T, h\in H$, and $m\in \tau^{-1}(\overline{N})$, 
the image $\tau(tmh^{-1})$ is contained in $\overline{N}$. But
$$
\tau(tmh^{-1})  = t m\theta_{an}(m) \theta(t)^{-1}.
$$ 
Since $\tau (m) = m \theta_{an}(m) \in \overline{N}$ and since $\overline{N}$ is $T*$-stable, 
the proof is finished.

\end{proof}

\begin{Remark}
Let $(T,B)$ be a pair of $\theta$-stable maximal torus and Borel subgroup such that $T\subseteq B$.
Let $\mc{V}$ denote the set of all $g\in G$ such that $\tau(g) \in N_G(T)$.
It is easy to verify that $\mc{V}\subset G$ is closed under the action of $T\times H$,   
$$
(t,h) \cdot g = tgh^{-1}\ \text{ for } t\in T, h\in H,\ g \in G.
$$
Let $V$ denote the set of $T\times H$-orbits in $\mc{V}$. For $v\in V$, let $x(v) \in \mc{V}$ denote 
a representative of the orbit $v$. The inclusion $\mc{V} \hookrightarrow G$ induces a bijection from $V$ 
onto the set of $B\times H$-orbits in $G$. In particular, $G$ is the disjoint union of the double cosets 
$B x(v) H$, $v\in V$. See~\cite{Helminck96}.
\end{Remark}

\begin{Theorem}\label{M_Q is spherical}
The following sets are in bijection with each other;
\begin{enumerate}
\item $B*$-orbits in $M_Q$, 
\item $T\times H$-orbits in $\tau^{-1} (\overline{N}) \subset M$.
\end{enumerate}
\end{Theorem}

\begin{proof}

We start with an observation; 
under $\tau$, the set of $B\times H$-orbits in $M$ is 
surjectively mapped onto the set of $B*$-orbits in $M_Q$.
To see this, first, we show that any $B\times H$-orbit in $M$ is mapped by $\tau$ 
onto a $B*$-orbit in $M_Q$. Let $\mc{O}_a$ be the $B\times H$-orbit of an element 
$a$ from $M$. Since
$$
\tau(bah^{-1}) = bah^{-1} \theta_{an}(bah^{-1}) = ba \theta_{an}(a) \theta_{an}(b) = b*\tau(a),
$$
we see that $\tau(\mc{O}_a)= B*\tau(a)$.
Next, we will show that any $B*$-orbit in $M_Q$ comes from a $B\times H$-orbit in $M$. 
Let $x$ be an element in $M_Q$. 
By Proposition \ref{B orbits intersect with closure(N)}, we know that any $B*$-orbit 
in $M_Q$ intersects $\overline{N}$. In particular, $B*x\cap \overline{N} \neq \emptyset$.
Let $\overline{n}$ be an element from $\overline{N}$ such that $b*x=\overline{n}$
for some $b\in B$. By Lemma~\ref{L:it acts}, we know that $T\times H$ acts on $\tau^{-1}(\overline{N})$.
Let $a\in \tau^{-1}(\overline{N})$ be such that $\tau(a) = \overline{n}$. Then 
the $B\times H$-orbit $\mc{O}_a$ of $a$ is mapped to $B*\tau(a)=B*x$. 
Now we know that $B\times H$-orbits in $M$ are mapped onto 
$B*$-orbits in $M_Q$. 

Incidentally, the argument in the above paragraph 
shows the following: the assignment defined by 
\begin{align}\label{A:define f}
f: (T\times H)a \longmapsto T*\tau(a) \longmapsto B*\tau(a)
\end{align}
is a surjective map between the set of $T\times H$-orbits in $\tau^{-1}(\overline{N})$
and the set of $B*$-orbits in $M_Q$. 
We proceed to show that $f$ is injective. 

Let $O$ be a $B*$-orbit in $M_Q$ and suppose that 
$\overline{n}_1$ and $\overline{n}_2$ are two elements 
from $\overline{N}\cap O$. Then there exists $b\in B$ such that 
$$
\overline{n}_1 = b*\overline{n}_2 = b \overline{n}_2 \theta_{an}(b).
$$
Since the Bruhat-Chevalley decomposition $M=\bigcup_{\dot{r}\in R} BrB$ is a disjoint union,
and $B$ is $\theta_{an}$-stable, 
we see from the uniqueness criterion that $b\in T$ and $\overline{n_2} \in \overline{n_1} T$. 
In other words, there exists $t\in T$ such that $t * \overline{n_2}= \overline{n_1}$. 
Let $a_1$ and $a_2$ be two element from $\tau^{-1}(\overline{N})$ such 
that $\tau(a_1)=\overline{n_1}$ and $\tau(a_2) = \overline{n_2}$. 
Then $t* \tau(a_2) = \tau(a_1)$. But $t*\tau(a_2) = \tau(ta_2)$, or, 
equivalently, $ta_2\in \tau^{-1}(\overline{n_1})$. 
Consequently, we see that the intersection with $\overline{N}$ 
of a $B*$-orbit $O$ ($=\tau(\mc{O}_{a_1})$) is covered by a single 
$T\times H$-orbit in $\tau^{-1}(\overline{N})$.
In particular, the map (\ref{A:define f}) is one-to-one. 
\end{proof}

\begin{Remark}\label{R:along a T orbit}
An important corollary of the proof of Theorem~\ref{M_Q is spherical}
is that the number of $B*$-orbits in $M_Q$ is finite. Indeed, any $B*$-orbit
in $M_Q$ intersects $\overline{N}$ along a $T*$-orbit and $\overline{N}/T$
is a finite semigroup. 
\end{Remark}

Now we are ready to prove our second main result, which states that 
the following sets are finite and they are in bijection with each other:
\begin{enumerate}
\item $B*$-orbits in $\overline{Q}$ (respectively, $B*$-orbits in $\overline{P}$), 
\item $T\times H$-orbits in $\tau^{-1} (\overline{N} \cap \overline{Q})$ (respectively, 
$T\times H$-orbits in $\tau^{-1} (\overline{N} \cap \overline{P})$).
\end{enumerate}

\begin{proof}[Proof of Theorem~\ref{T:Corollary}.]
Since $M_Q$ is closed, the inclusions $P\subseteq Q \subseteq M_Q$ imply 
that $\overline{P}\subseteq \overline{Q} \subseteq M_Q$. Moreover, we know 
that $P$ and $Q$, and hence $\overline{P}$ and $\overline{Q}$ are $B*$-stable.
By Remark~\ref{R:along a T orbit} we know that $M_Q$ is comprised of 
finitely many $B*$-orbits. The rest of the proof follows from Theorem~\ref{M_Q is spherical}.
\end{proof}

There is a well known classification, due to Cartan, of involutions on semisimple groups.
For the classical groups, up to inner automorphisms there are seven types of involutions in total. 
For the exceptional groups there are in total nine involutions. 
See Chapter X, Section 6 of~\cite{Helgason78} for a complete list. 
We finish this section by presenting some examples.

\begin{Example}
	
Let $G_0$ denote $\mt{PSL}_n$, the projective special linear group of $n\times n$ matrices with determinant 1. 
Then $\theta_0 : G_0 \rightarrow G_0$ defined by $\theta_0(g) = (g^{-1})^\top$ ($g\in G_0$) is an involutory automorphism. 
Let $T_0$ denote the maximal torus of diagonal matrices in $G_0$ and let $\omega_1$ denote the first fundamental 
weight. Let $(\rho_0,V) \cong (id,\C^n)$ denote the corresponding 
irreducible (minuscule) representation. Then the $J$-irreducible monoid associated with $\omega_1$ 
is nothing but the monoid of $n\times n$ matrices, 
$$
Z_V:=\overline{\C^* \cdot \mt{PSL}_n} = \mt{Mat}_n,
$$
which we denote by $M$. Then the unit group of $M$ is $G=\mt{GL}_n$.
Clearly, $\theta_0$ extends to $G$ by the same formula, $\theta(g)  = (g^{-1})^\top$ ($g\in G$). 
The $G*$-orbit of the identity is equal to the set of invertible symmetric $n\times n$ matrices, 
$$
G*1_{\mt{GL}_n}=P =\{ gg^\top:\ g\in \mt{GL}_n\}.
$$

We observe that for our choices of $\theta$ and $G_0$, the subvariety $Q:=\{ g\in G:\ \theta(g) =g^{-1}\}$ is equal to $P$.
Therefore, in $M$, we have 
$$
\overline{Q}= \overline{P} = \mt{Sym}_n,
$$
the affine variety of symmetric $n\times n$ matrices. Also, we notice that the unique antiinvolution on $M$ 
that is extended from the involution $\theta$ on $G$ is given by $\theta_{an} (m) = m^\top$ for $m\in M$. 
Therefore, $M_Q$ is equal to $\mt{Sym}_n$ as well. 
Finally, we know from~\cite{Szechtman07} that $B*$-orbits in $\mt{Sym}_n$ are parametrized 
by the $n\times n$ ``partial involutions'' in the ``rook monoid'' $R_n$.
Here, the {\em rook monoid} is the Renner monoid of $\mt{Mat}_n$; it is the finite monoid which consists 
of $n\times n$ 0/1 matrices with at most one 1 in each row and column. A {\em partial involution} in $R_n$
is an element $x\in R_n$ such that $x^\top = x$.

\end{Example}

\begin{Example}
Let $(\rho_0,V)$ denote the second fundamental representation $V=\bigwedge^2 \C^{2n}$
of $G_0:=\mt{PSL}_{2n}$. As before, let $T_0$ denote the maximal torus consisting of diagonal
matrices in $G_0$. We consider the involution $\theta_0(g) = -J (g^{-1})^\top J$ ($g\in G_0$), 
where $J$ is the $2n\times 2n$ block diagonal matrix 
$$
J=\mt{diag}(J_2,\dots,J_2)\ \text{ with } J_2= \begin{pmatrix} 0 & 1 \\ -1 & 0 \end{pmatrix}.
$$
More explicitly, $V$ is equal to the space of $2n\times 2n$ skew-symmetric matrices, and the 
action of $G_0$ on $V$ is given by 
$$
g \cdot A = (g^{-1})^\top A g^{-1}.
$$
	
For the notational ease, let us denote the operator $\rho_0(g)$ on $V$ ($g\in G_0$) by $\phi_g$. 
Note that, for $g=J$ we have $\phi_J^2 = 1_{\mt{GL}(V)}$.
It is not difficult to show that $\rho_0$ is faithful, and that the extension of $\theta_0$ 
to $G= \C^*\cdot \rho_0(G_0)$ is given by 
$$
\theta( c\phi_g ) = c\phi_J \phi_{(g^{-1})^\top} \phi_J, \text{ for all $g\in G_0$ and $c\in \C^*$.}
$$

Now, let $y$ be an element from $Q$. 
If $y=c\phi_g$ for some $g\in G_0$, and $c\in \C^*$, then 
$$
c^{-1}\phi_g^{-1}= \theta(y) = c \phi_J \phi_{(g^{-1})^\top}  \phi_J,\ \text{ or, equivalently }\ \mt{1}_V =  c^2\phi_{gJ (g^{-1})^\top J}.
$$	
Since $\rho_0$ is faithful, $c=1$, and $gJ (g^{-1})^\top J=1_{G_0}$, or $g^{-1}= J(g^{-1})^\top J$. 
In other words, $Q$ is isomorphic to $Q_0:=\{ g\in \mt{PSL}_{2n}:\ \theta_0(g) = g^{-1} \}$. 

On the other hand, we know that $Q_0$ is equal to $P_0 := \{  g\theta_0(g^{-1}):\ g \in \mt{PSL}_{2n}\}$,
see Section 11.3.5 of \cite{GoodmanWallach}. 
Since the image of $P_0$ under $\rho_0$ is equal to $P$, we see that $P=Q$, 
so $P$ is a closed subvariety of $G$. 	

Let $\theta_{an}$ be the unique antiinvolution extension of $\theta$ to the monoid $M$ of $(\rho_0,V)$.
Then 
$$
\overline{Q} = M_Q= \{ y\in M :\ \theta_{an}(y) = y \}.
$$
Next, we compute the parametrizing set of $B*$-orbits in $M_Q= M_P$. 
To this end, we determine the normalizer of $T$ in $G$.  
We claim that $N_G(T)=\C^* \rho_0(N_{G_0}(T_0))$. 
Indeed, let $x=c\rho_0(g)\in G$ be an element from the normalizer of $T$, 
and let $t\in T$. Since $t= d\rho_0(t')$ for some $t'\in T_0$ and $d\in \C^*$, we have 
$x t x^{-1}= d \rho_0(g t' g^{-1})\in T$, or equivalently, $gt'g^{-1}\in T_0$.
Thus, $t$ lies in $\C^* \rho_0(N_{G_0}(T_0))$. The converse inclusion is obviously true. 

Let us look at a typical element of $N_G(T)$.
Assume that $g$ is a monomial matrix, that is to say, every row and every column have exactly one nonzero entry. 
We will prove that, once a basis is fixed, $\rho_0(g)=\phi_g$ is a monomial matrix as well.  
Towards this end, we choose the following basis 
$$
F_{i,j}= E_{i,j}- E_{j,i} \qquad 1\leq i < j \leq 2n,
$$
where $E_{i,j}$'s are the elementary matrices. 
Suppose that the inverse of $g\in G_0$ is the matrix $g^{-1}= (g_{k,l})_{k,l=1}^{2n}$. 
Obviously, $g^{-1}$ is a monomial matrix, as well. 
Since $\rho_0(g) \cdot E_{i,j} = (g^{-1})^\top E_{i,j} g^{-1}= (g_{i,k}g_{j,l})_{k,l}^{2n}$, we see that 
\begin{align}\label{A:explicitly}
g\cdot F_{i,j}= \rho_0(g) \cdot F_{i,j} &= (g_{i,k}g_{j,l} - g_{j,k}g_{i,l})_{k,l}^{2n}.
\end{align}
%	Also, since transposing commutes with taking inverses, we have 
%	\begin{align*}
%	\rho_0(g^\top) \cdot F_{i,j} &= -(g_{i,k}g_{j,l} - g_{j,k}g_{i,l})_{k,l}^{2n}.
%	\end{align*}

We continue with a special case of our claim by assuming that $g$ is a diagonal matrix. 
Then the $(k,l)$-th entry (with $k< l$) of $g \cdot F_{i,j}$ is nonzero if and only if $i=k$ and $j=l$. 
In this case, $g \cdot F_{i,j} = g_{i,i}g_{j,j} F_{i,j}$. Thus, the matrix representing $\rho_0(g)$ is 
the $n(n-1)\times n(n-1)$ diagonal matrix $\mt{diag}( s_{1,2},s_{1,3},\dots, s_{n-1,n})$ with $s_{i,j}=g_{i,i}g_{j,j}$.
Now, more generally, assume that $g$ is a monomial matrix. Then the $(k,l)$-th entry (with $k< l$) 
$g_{i,k}g_{j,l} - g_{j,k}g_{i,l}$ of $g \cdot F_{i,j}$ is nonzero if and only if one of the following is true;

\begin{itemize}
\item[i)] the entries of $g^{-1}$ at its $(i,k)$-th and the $(j,l)$-th positions are nonzero at the same time, or 
\item[ii)] the entries of $g^{-1}$ at its $(i,l)$-th and the $(j,k)$-th positions are nonzero at the same time.
\end{itemize}
Observe that i) and ii) do not hold true at the same time. Observe also that, for each $i<j$ there 
exists a unique pair $(k,l)$ with $k<l$ such that either i) is true, or ii) is true. 
Therefore, if $g^{-1}$ is a monomial matrix, then 
\begin{align}\label{monomial matrix action}
g \cdot F_{i,j} = 
\begin{cases} 
g_{i,k}g_{j,l} F_{k,l} & \text{ if }  g_{i,k}g_{j,l} \neq 0, \\
g_{i,l}g_{j,k} F_{k,l} & \text{ if }  g_{i,l}g_{j,k} \neq 0.
\end{cases}
\end{align}
It follows that if $g$ is a monomial matrix, then so is the matrix of $\rho_0(g)=\phi_g$.

Now, let $x\in M$ be an element from $\overline{N_G(T)}$. Since the elements of $\overline{N_G(T)}$ 
are obtained from those of $N_G(T)$ by taking limits (in the algebraic sense), we see $x\cdot F_{i,j}$ 
is either identically zero, or it is a scalar multiple of $F_{k,l}$ for some $k,l$ as in (\ref{monomial matrix action}).
In other words, $x$ is obtained from the image of a monomial matrix in $G_0$ by replacing some of its entries by zeros.

It is well know that the invertible symmetric monomial matrices modulo the maximal torus of diagonal matrices 
represent the $B*$-orbits in $Q$, and furthermore, the finite set of orbit representatives 
is in bijection with the fixed point free involutions of the symmetric group $S_{2n}$ (see~\cite{RS90}).
Thus, in our case, the representing matrices are those that are obtained from the fixed point free monomial 
matrices by replacing some of the nonzero entries by zeros. 
These are precisely the ``partial fixed point free involutions,'' introduced in~\cite{Cherniavsky2011}.
%See also ~\cite{CanCherniavskyTwelbeck}.
	
\end{Example}

\section{Final remarks}\label{S:Final}

Given a reductive monoid $M$ with an antiinvolution $\theta_{an}$, 
we now have the notion of a {\em symmetric submonoid}
\begin{align}
M_{{an}} :=\{ m\in M : m \theta_{an} (m) = 1_M \}.
\end{align}
Observe that the identity element $1_M$ of $M$ is the identity element $1_G$ of $G$. 
Therefore, $ \theta_{an}(1_M)=\theta_{an}(1_G)= \theta_{an}(1_G) \theta_{an}(1_G)$, 
hence $\theta_{an}(1_M)=1_M$.
In other words, $1_M\in M_{an}$. Also, if $m_1,m_2\in M_{{an}}$, then 
$$
m_1m_2\theta_{an}(m_1m_2) = m_1 m_2\theta_{an}(m_2)\theta_{an}(m_1) 
= m_1\cdot 1_M\cdot \theta_{an}(m_1) = 1_M.
$$
Therefore, $m_1m_2 \in M_{an}$.
Note that if an element $g\in G$ lies in $M_{an}$, then 
$1_G= g^{-1} \theta_{an}(g^{-1})= g^{-1} \theta(g)$, hence $\theta(g) = g$. 
In other words, the group of invertible elements of $M_{an}$ is the fixed subgroup $H=G^{\theta}$.
The above argument provides us with an effective way of producing new linear algebraic monoids,
one for each antiinvolution $\theta_{an}$ on $M$.

\bibliography{References.bib}
\bibliographystyle{plain}

\end{document}